\documentclass[12pt]{amsart}
\usepackage{amsfonts,amssymb}
\usepackage[all]{xypic}
\usepackage{amsmath}
\usepackage{amsthm,amscd,latexsym}

\newtheorem{theorem}{Theorem}[section]
\newtheorem{corollary}[theorem]{Corollary}
\newtheorem{Definition}[theorem]{Definition}
\newtheorem{lemma}[theorem]{Lemma}
\newtheorem{proposition}[theorem]{Proposition}
\newtheorem{Example}[theorem]{Example}
\newtheorem{Remark}[theorem]{Remark}

\newenvironment{remark}{\begin{Remark}\begin{em}}{\end{em}\end{Remark}}

\newenvironment{definition}{\begin{Definition}\begin{em}}{\end{em}\end{Definition}}

\newcommand{\R}{{\mathbb R}}

\newcommand{\PP}{\mathcal P}
\newcommand{\pp}{\mathbb P}

\newcommand{\G}{\frak{G}}

\newcommand{\Bw}{{\mathbf w}}

\begin{document}
\title[Monotonicity]{Monotonic Properties of the Least Squares Mean}
\author{Jimmie Lawson and Yongdo Lim}

\address{Department of Mathematics, Louisiana State University,
Baton Rouge, LA70803, USA}\email{lawson@math.lsu.edu}

\address{ Department of Mathematics,
Kyungpook National University, Taegu 702-701, Korea}
\email{ylim@knu.ac.kr}

\maketitle
\begin{abstract}
We settle an open problem of several years standing by showing that the least-squares
mean for positive definite matrices is monotone for the usual (Loewner) order.  Indeed
we show this is a special case of its appropriate generalization to partially ordered
complete metric spaces of nonpositive curvature. Our techniques extend to establish other
basic properties of the least squares mean such as continuity and joint concavity.  Moreover,
we introduce a weighted least squares means and extend our results to this setting.
\end{abstract}

\section{Introduction}
Not only does the  study of positive definite matrices remain a
flourishing area of mathematical investigation (see e.g., the recent
monograph of R.\ Bhatia \cite{Bh07} and references therein), but
positive definite matrices have become fundamental computational
objects in many areas of engineering, statistics, quantum
information, and applied mathematics. They appear as covariance
matrices in statistics, as elements of the search space in convex
and semidefinite programming, as kernels in machine learning,  as
density matrices in quantum information, and as diffusion tensors in
medical imaging, to cite a few. A variety of metric-based
computational algorithms for positive definite matrices have arisen
for approximations, interpolation, filtering, estimation, and
averaging, the last being the concern of this paper. In recent
years, it has been increasingly recognized that the Euclidean
distance is often not the most suitable for the set  of positive
definite matrices--the positive  symmetric  cone $\pp=\pp_m$ for
some $m$--and that working with the proper geometry does matter in
computational problems.  It is thus not surprising that there has
been increasing interest in the trace metric, the distance metric
arising from the natural Riemannian structure on $\pp$ making it a
Riemannian manifold, indeed a symmetric space, of negative
curvature. (Recall the trace metric distance between two positive
definite matrices is given by
$\delta(A,B)=(\sum_{i=1}^{k}\log^{2}\lambda_{i}(A^{-1}B))^{\frac{1}{2}}$,
 where $\lambda_{i}(X)$ denotes the $i$-th eigenvalue of $X$ in non-decreasing order.)
Recent contributions that have advocated
the use of this metric in applications include \cite{FJ,PFA,ZM} for tensor computing in medical
imaging and \cite{Ba08}  for radar processing.

Since the pioneering paper of Kubo and Ando \cite{KA80}, an
extensive theory of two-variable means has sprung up for positive
matrices and operators, but the multivariable case for $n>2$ has
remained problematic.  Once one realizes, however, that the matrix
geometric mean
$\G_2(A,B)=A\#B:=A^{1/2}(A^{-1/2}BA^{-1/2})^{1/2}A^{1/2}$ is the
metric midpoint of $A$ and $B$ for the trace metric (see, e.g.,
\cite{LL01,Bh07}), it is natural to use an averaging technique over
this metric to extend this mean to a larger number of variables.  First
M.\ Moakher \cite{Mo05} and then Bhatia and Holbrook \cite{BH06},
\cite{BH06b} suggested extending the geometric mean to $n$-points by
taking the mean to be the unique minimizer of the sum of the squares
of the distances:
$$\G_n(A_1,\ldots,A_n)=\hbox{arg}\min_{X\in\pp}\sum_{i=1}^n \delta^2(X,A_i).$$
This idea had been anticipated by  \'Elie Cartan (see, for example, section 6.1.5
of \cite{Be03}), who showed among other things such a unique minimizer exists
if the points all lie in a convex ball in a Riemannian
manifold, which is enough to deduce the existence of the
least squares mean globally for ${\mathbb P}$.

Another approach, independent of metric notions,  was suggested by
Ando, Li, and Mathias \cite{ALM04} via a ``symmetrization procedure"
and induction. The Ando-Li-Mathias paper was also important for
listing, and deriving for their mean, ten desirable properties for
extended geometric means $g:{\mathbb P}^n\to {\mathbb P}$ that one
might anticipate from properties of the two-variable geometric mean,
where ${\Bbb P}={\Bbb P}_{m}$ denotes the convex cone of $m\times m$
positive definite Hermitian matrices equipped with  the Loewner
order $\leq$. For ${\Bbb A}=(A_{1},\dots,A_{n}), {\Bbb
B}=(B_{1},\dots,B_{n})\in {\Bbb P}^{n}, \sigma\in S^{n}$ a
permutation on $n$-letters, $ {\bf a}=(a_{1},\dots,a_{n})\in {\Bbb
R}_{++}^{n}\ ({\mathbb R}_{++}=(0,\infty))$, these are
\begin{itemize}
\item[(P1)]$($Consistency with scalars$)$
$ g({\Bbb A})=(A_{1}\cdots A_{n})^{1/n}$ if the $A_{i}$'s commute;
\item[(P2)] $($Joint homogeneity$)$
$ g(a_{1}A_{1},\dots,a_{n}A_{n})= (a_{1}\cdots a_{n})^{1/n}g({\Bbb
A});$
\item[(P3)] $($Permutation invariance$)$
$g({\Bbb A}_{\sigma}) =g({\Bbb A}),$ where ${\Bbb
A}_{\sigma}=(A_{\sigma(1)},\dots,A_{\sigma(n)});$
\item[(P4)] $($Monotonicity$)$ If $B_{i}\leq A_{i}$ for all $1\leq i\leq n,$ then
$ g({\Bbb B})\leq g({\Bbb A});$
\item[(P5)] $($Continuity$)$ $g$ is continuous;
\item[(P6)] $($Congruence invariance$)$
$g(M^{*}{\Bbb A}M)= M^{*}g({\Bbb A})M$ for  invertible invertible
matrix $M,$ where
$M(A_{1},\dots,A_{n})M^{*}=(MA_{1}M^{*},\dots,MA_{n}M^{*});$
\item[(P7)] $($Joint concavity$)$
$g(\lambda {\Bbb A}+(1-\lambda){\Bbb B})\geq \lambda g({\Bbb
A})+(1-\lambda)g({\Bbb B})$ for $0\leq \lambda \leq 1$;
\item[(P8)] $($Self-duality$)$
$g(A_{1}^{-1} ,\dots, A_{n}^{-1})^{-1}= g(A_{1} ,\dots, A_{n});$
\item[(P9)]$($Determinental identity$)$
$ {\mathrm {Det}}g({\Bbb A})= \prod_{i=1}^{n}({\mathrm
{Det}}A_{i})^{1/n}$; and
\item[(P10)] $($AGH mean inequalities$)$
$ n(\sum_{i=1}^{n}A_{i}^{-1})^{-1}\leq g({\Bbb A}) \leq
\frac{1}{n}\sum_{i=1}^{n}A_{i}.$
\end{itemize}
We call a mean $g$ of $n$-variables satisfying these properties a
\emph{symmetric geometric mean}, the adjective ``symmetric" describing
its invariance under permutations, property (P3).



 The Ando-Li-Mathias mean proved to be computationally cumbersome, and
 Bini, Meini, and Poloni \cite{BMP10} suggested an alternative with more rapid
 convergence properties, which also satisfied the ten axioms.  One notes in particular
 that while the axioms characterize the two-variable case, this is no longer true
 in the $n$-variable case, $n>2$.

These ten properties may be generalized to the setting of weighted
geometric means.  We recall that the two-variable weighted geometric mean is given by 
$$t\mapsto \G_2(1-t,t;A,B)=A\#_{t}B:=:A^{1/2}(A^{-1/2}BA^{-1/2})^{t}A^{1/2}, $$
which is a geodesic parametrization of the unique geodesic passing through  $A$ and $B$ for $A\ne B$.
A weighted geometric mean  of $n$-positive definite matrices should be defined for each weight, 
where the weights $\omega=(w_1,\ldots, w_n)$ vary
over  $\Delta_{n},$ the simplex of positive probability vectors convexly spanned by the unit coordinate vectors.
We define a \emph{weighted geometric mean} of $n$ positive
definite matrices to be  a map $g:\Delta_{n}\times {\Bbb P}^{n}\to {\Bbb
P}$ satisfying the following properties:
\begin{itemize}
\item[(P1)] $($Consistency with scalars$)$
$ g(\omega;{\Bbb A})=A_{1}^{w_{1}}\cdots A_{n}^{w_{n}}$ if the $A_{i}$'s
commute;
\item[(P2)] $($Joint homogeneity$)$
$ g(\omega;a_{1}A_{1}, \dots, a_{n}A_{n})= a_{1}^{w_{1}}\cdots
a_{n}^{w_{n}}g(\omega;{\Bbb A});$
\item[(P3)] $($Permutation invariance$)$
$ g(\omega_{\sigma};{\Bbb A}_{\sigma}) =g(\omega;{\Bbb A})$, where
$\omega_{\sigma}=(w_{\sigma(1)},\dots,w_{\sigma(n)});$
\item[(P4)] $($Monotonicity$)$ If $B_{i}\leq A_{i}$ for all $1\leq i\leq n,$ then
$g(\omega;{\Bbb B})\leq g(\omega;{\Bbb A});$
\item[(P5)] $($Continuity$)$ The map $g(\omega;\cdot)$ is
continuous;
\item[(P6)] $($Congruence invariance$)$
$g(\omega;M^{*}{\Bbb A}M)= M^{*}g(\omega;{\Bbb A})M$ for  any
invertible $M;$
\item[(P7)] $($Joint concavity$)$ $g(\omega;\lambda {\Bbb A}+(1-\lambda){\Bbb B})\geq \lambda g(\omega;{\Bbb
A})+(1-\lambda)g(\omega;{\Bbb B})$ for $0\leq \lambda \leq 1$;
\item[(P8)] $($Self-duality$)$
$ g(\omega;A_{1}^{-1} ,\dots, A_{n}^{-1})^{-1}= g(\omega;A_{1}
,\dots, A_{n});$
\item[(P9)]$($Determinental identity$)$
$ {\mathrm {Det}}g(\omega;{\Bbb A})= \prod_{i=1}^{n}({\mathrm
{Det}}A_{i})^{w_{i}};$ and
\item[(P10)] $($AGH weighted  mean inequalities$)$
$ (\sum_{i=1}^{n}w_{i}A_{i}^{-1})^{-1}\leq g(\omega;{\Bbb A}) \leq
\sum_{i=1}^{n}w_{i}A_{i}.$
\end{itemize}
We note that the two-variable weighted geometric mean
$\G_2(1-t,t;A,B)=A\#_{t}B, \ t\in [0,1],$ satisfies $(P1)-(P10).$

In their study of the symmetric least squares mean, Moakher \cite{Mo05} and Bhatia and Holbrook \cite{BH06}, \cite{BH06b}
have derived for it some of the axiomatic properties (P1)-(P10) satisfied by
the Ando-Li-Mathias geometric mean: consistency with scalars, joint
homogeneity, permutation invariance, congruence invariance, and self-duality (the last two being
true since congruence transformations and inversion are isometries).
Further, based on computational experimentation,
Bhatia and Holbrook conjectured monotonicity for the least squares mean (problem
19 in  ``Open problems in matrix theory" by X.\ Zhan \cite{Zh08}).
Providing a positive solution (Corollary \ref{C:mono}) to this conjecture was the original
motivation for this paper.

 In this paper we introduce the \emph{weighted least squares mean}
${\frak G}_{n}(\omega;A_{1},\dots,A_{n})$ of $(A_{1},\dots,A_{n})$
with the weight $\omega=(w_{1},\dots,w_{n})\in \Delta_{n}$, which  is
defined to be
\begin{equation}\label{E:1.1}
{\frak G}_{n}(\omega;A_{1},\dots,A_{n})=\mbox{arg}\min_{X\in {\Bbb
P}}\sum_{i=1}^n w_{i} \delta^2(X,A_i).
\end{equation}
 Computing appropriate derivatives as in
(\cite{Bh07,Mo05}) yields that the weighted least squares mean
coincides with the unique positive definite solution of the equation
\begin{equation}\label{E:least}\sum_{i=1}^{n}w_{i}\log(XA_{i}^{-1})=0.
\end{equation}
It is not difficult to see from (\ref{E:1.1}) and (\ref{E:least}) and some elementary
facts about matrices and the trace metric that the weighted least squares mean satisfies
$(P1)-(P3), (P6), (P8)$ and $(P9).$ In this paper we show that
the weighted least squares mean satisfies all the properties
$(P1)-(P10)$ by verifying all the additional properties (P4), (P5), (P7), and (P10).
As far as we know, this is the first verification of properties (P4) and (P7)
in both the weighted and unweighted cases and of (P10)
in the weighted case, the unweighted case having been shown by Yamazaki (\cite{Ya}).
 We thus see that the (weighted) least squares mean provides
another important example of a (weighted) geometric mean.   We
further show that the weighted least squares mean is non-expansive:
$\delta({\frak G}_{n}(\omega;A_{1},\dots,A_{n}),
{\frak G}_{n}(\omega;B_{1},\dots,B_{n}))\leq
\sum_{i=1}^{n}w_{i}\delta(A_{i},B_{i}).$

The main tools of the paper involve the theory of nonpositively curved metric spaces and techniques
from probability and random variable theory and the recent combination of the two, particularly
by K.-T. Sturm \cite{St03}.  Not only are these tools crucial for our developments, but also, we believe,
significantly enhance the potential usefulness of the least squares mean.

\section{Metric Spaces and Means}
The setting appropriate for our considerations is
that of \emph{globally nonpositively curved metric spaces}, which we call \emph{NPC spaces} spaces for short
(since we do not consider the locally nonpositively curved spaces).
These are complete metric spaces $M$ satisfying for all $x,y\in M$, there exists $m\in M$ such that for all $z\in M$
\begin{equation}\label{E:2.1}
d^2(m,z)\leq \frac{1}{2}d^2(x,z)+\frac{1}{2}d^2(y,z)-\frac{1}{4}d^2(x,y).
\end{equation}
Such spaces are also called (global) CAT$_0$-spaces or Hadamard spaces.
The theory of such spaces is quite extensive; see, e.g., \cite{Ba95}, \cite{BrH99}, \cite{Jo97}, \cite{St03}.
In particular the $m$ appearing in (\ref{E:2.1}) is the unique metric midpoint between
$x$ and $y$. By inductively choosing midpoints for dyadic rationals and extending
by continuity, one obtains for each $x\ne y$ a unique metric minimal geodesic $\gamma:[0,1] \to M$
satisfying $d(\gamma(t),\gamma(s))=\vert t-s\vert d(x,y)$.  We denote $\gamma(t)$ by
$x\#_t y$ and call it the $t$-\emph{weighted mean} of $x$ and $y$.  The midpoint $x\#_{1/2}y$
we denote simply as $x\#y$.  We remark that by uniqueness $x\#_t y=y\#_{1-t} x$; in particular,
$x\#y=y\#x$.
\begin{remark}\label{R:2.1}
Equation (\ref{E:2.1}) is sometimes referred to as the \emph{semiparallelogram law}, since it can derived from the parallelogram
law in Hilbert spaces by replacing the equality with an inequality (see \cite{LL01}).  It is satisfied by the length metric in any simply
connected nonpositively curved  Riemannian manifold \cite{Lang}.  Hence the metric definition represents a metric generalization of
nonpositive curvature.  The trace metric on the Riemannian symmetric space of positive definite matrices is a particular example
(\cite{Lang,LL01}).
\end{remark}

Equation (\ref{E:2.1}) admits a more general formulation in terms of the weighted mean (see e.g. \cite[Proposition 2.3]{St03}). For
all $0\leq t\leq 1$ we have
\begin{equation}\label{E:2.2}
d^2(x\#_t y,z)\leq (1-t)d^2(x,z)+td^2(y,z)-t(1-t)d^2(x,y).
\end{equation}

An $n$-\emph{mean} on a set $X$ is a function $\mu:X^n\to X$ satisfying the idempotency law
$\mu(x, x,\ldots,x)=x$.   It is \emph{symmetric} if it is invariant under all permutations $\sigma$ of
$\{1,\ldots,n\}$, i.e., $\mu(x_1,\ldots, x_n)=\mu(x_{\sigma(1)}, \ldots, x_{\sigma(n)})$.
For a metric space $X$ with weighted mean, the
operation $x\#_t y$ is a $2$-mean for each $t$.
A special case is the midpoint mean $x\# y$ for $t=1/2$, which is symmetric.

The problem of extending the geometric mean of two positive definite
matrices to an $n$-variable mean for $n\geq 3$ generalizes to the
setting of metric spaces with weighted means.  Under appropriate
metric hypotheses, all of which are implied by the NPC condition,
the symmetrization procedure applies and inductively yields
multivariable means extending $x\#y$ for each $n\geq 3$; see
Es-Sahib and Heinrich \cite{EH99} and the authors \cite{LL08}. The
weighted $2$-means and the mean of Bini, Meini, and Poloni
\cite{BMP10} also generalize to  NPC-spaces, and even weaker metric
settings \cite{LLL10}.

The least squares mean can be
immediately formulated in any metric space $(M,d)$:
\begin{equation}\label{E:2.3}
{\frak G}_n(a_1,\ldots, a_n)=\mbox{arg}\min_{z\in M}\sum_{i=1}^n d^2(z,a_i).
\end{equation}
In general this mean is not defined, since the minimizer may fail to exist or fail to be unique.
One also has a weighted version of the mean. Given
$(a_1,\ldots,a_n)\in M^n$, and positive real numbers $w_1,\ldots,w_n$ summing to $1$,
we define
\begin{equation}\label{E:2.3b}
{\frak G}_n(w_1,\ldots, w_n;a_1,\ldots, a_n):=\mbox{arg}\min_{z\in M}\sum_{i=1}^n w_id^2(z,a_i).
\end{equation}
provided the minimizer exists and is unique.  As mentioned previously, it was shown by
E.\ Cartan (see  \cite{Be03}) that this is
the case if the points all lie in a convex ball in a Riemannian
manifold. For our purposes, we note that existence and uniqueness holds in general for NPC
spaces as can be readily deduced from the uniform convexity of the
metric; see \cite[Propositions 1.7, 4.3]{St03}.  Note that the mean in
equation (\ref{E:2.3}) is a symmetric mean and the one in
(\ref{E:2.3b}) is permutation invariant in the sense of
 of property (P3) for weighted means given in the Introduction.
By taking $w_i=1/n$ for each $i=1,\ldots,n$, we see that the former mean is a special case of the latter,
so we work with the weighted case in what follows.
Although this mean is sometimes referred to as the Karcher mean in light of its
appearance in his work on Riemannian manifolds \cite{Ka77}, we will
refer to it as the \emph{weighted least squares mean}, or simply as the \emph{least squares mean}.

\section{The method}
Since our method of proof in this paper departs rather radically
from previous approaches to the theory of matrix means, we judge
that it is worthwhile to give a quick,  informal, and intuitive
overview of our approach and methods.  Suppose that we are given an
NPC metric space $(M,d)$, a tuple $(a_1,\ldots, a_n)\in M^n$, and a
weight $(w_1,\ldots, w_n)$ of positive real numbers. We imagine
carrying out a sequence of independent trials in which we randomly
choose in each trial an integer from the set  $\{1,\ldots,n\}$ in
such a way that $i$ is chosen with probability $w_i$.  If $i_k$ is
chosen on the $k^{th}$-trial, then we set $x_k=a_{i_k}$.   We define
a ``random walk" $\{s_k\}$ using this data by setting $s_1=x_1$,
$s_2=s_1\#x_2$, $s_3=s_2\#_{1/3} x_3$, and in general
$s_k=s_{k-1}\#_{1/k} x_k$, that is, at stage $k$ we move from
$s_{k-1}$ toward $x_k$ a fraction of $1/k$ of the distance between
them.  It is then a remarkable consequence of Sturm's Theorem 4.7 of
\cite{St03} that as we run through all possible outcomes of this
procedure, almost always the sequence $\{s_k\}_{k\in{\mathbb N}}$
will converge to ${\frak G}_n(w_1,\ldots,w_n;a_1,\ldots,a_n)$, the
weighted least squares mean.

This machinery provides a powerful tool for the study of the least squares mean.  Many properties of the
weighted $2$-means can be shown to extend to their finite iterations $s_n$, as defined in the previous paragraph,
and then shown to be preserved in passing to the limit, the weighted least squares mean.
It could also potentially be a useful computational tool to approximate
the weighted least squares mean by simulating the preceding random walk up through some stage $s_n$ for
large enough $n$.

For an $k$-tuple $(x_1,\ldots, x_k)\in M^{k}$, we can compute $s_k$
as defined in the first paragraph and use this value to define a
mean $S_k(x_1,\ldots,x_k)=s_k$.  In \cite{St03} Sturm has called
this mean the \emph{inductive mean} for NPC spaces, a mean which
appeared earlier in \cite{ST,AKL} for positive definite matrices.
Its explicit definition is given inductively  by $S_2(x,y)=x\# y$
and for $k\geq 3$, $S_{k}(x_1,\ldots, x_{k})=S_{k-1}(x_1,\ldots,
x_{k-1})\#_{\frac{1}{k}}x_{k}$.

\section{Random variables and barycenters}
In recent years significant portions of the classical theory of
real-valued random variables on a probability space have been
successfully generalized to the setting in which the random variables take values
in a metric space $M$. We quickly recall some of this theory as worked out, for
example, by Es-Sahib and Heinrich \cite{EH99} and particularly by Sturm \cite{St03}.

Let $(\Omega,\mathcal{A},\sigma)$
be a probability space: a set $\Omega$ equipped with a $\sigma$-algebra $\mathcal{A}$ of subsets,
and a $\sigma$-additive probability measure $\sigma$ on $\mathcal{A}$. We typically write the measure or
probability of $A\in\mathcal{A}$ by $P(A)$ instead of $\sigma(A)$.  For a metric space $(M,d)$,
an $M$-\emph{valued random variable} is a function $X:\Omega\to M$ which is measurable in the sense
that $X^{-1}(B)\in\mathcal{A}$ for every Borel subset of $M$.  We further impose the technically useful assumption
that the image $X(\Omega)$ is a separable subspace of $M$.

The push-forward of the measure $\sigma$ by $X$ is denoted $q_X$ and defined by $q_X(B)=\sigma(X^{-1}(B))$ for
each Borel subset $B$ of $M$.  It is a probability measure on the Borel sets of $M$ and is called the \emph{distribution}
of $X$.  A sequence of random variables $\{X_n\}$ is \emph{identically distributed} (i.d.) if all have the same distribution.
For any $q_X$-integrable function $\phi:M\to\R$, one has the basic formula $\int_M \phi \,dq_X=\int_\Omega \phi X
\, d\sigma$.

A collection of random variables $\{X_i:i\in I\}$ is \emph{independent} if for every finite $F\subseteq I$ of cardinality
at least two, $P(\bigcap_{i\in F} X_i^{-1}(B_i))=\prod_{i\in F}P(X_i^{-1}(B_i))$, where $\{B_i:i\in I\}$ is a collection
of Borel subsets of $M$.  A sequence $\{X_n\}$ is i.i.d. if it is both independent and identically distributed.

Assume henceforth that $M$ is an NPC-space.  Let $\PP(M)$ denote the set of
probability measures with separable support on $(M,\mathcal{B}(M))$,
where $\mathcal{B}(M)$ is the collection of Borel sets. We define the
collection $\PP^\theta(M)$ of probability measures $q\in\PP(M)$
to be those satisfying $\int_M
d^\theta(z,x)q(dx)<\infty$ for all $z\in M$. Members of $\PP^1(M)$
are called \emph{integrable} and those in $\PP^2(M)$ are called
\emph{square integrable}. We define a random variable $X:\Omega\to
M$ to be in $L^\theta$ if its distribution $q_X\in\PP^\theta(M)$. In
particular, it is integrable if $\int_\Omega
d(z,X(\omega))\sigma(d\omega)=\int_M d(z,x) q_X(dx)<\infty$ for
$z\in M$.  We define a sequence $\{X_n\}$ of random variables to be
\emph{uniformly bounded} if there exists $z\in M$ and $R>0$ such
that the image of each $X_n$ lies within the ball around $z$ of
radius $R$.

Following Sturm \cite{St03}, we define the \emph{barycenter} $b(q)$ of $q\in\PP^1(M)$ by
\begin{equation}\label{E:3.1}
b(q)=\mbox{arg}\min_{z\in M}\int_M [d^2(z,x)-d^2(y,x)]q(dx).
\end{equation}
Sturm uses the uniform convexity of the metric to show that
independent of $y$ there is a unique $z=b(q)$, the barycenter (by
definition), at which this minimum is obtained \cite[Proposition
4.3]{St03}, and that for the case that $q$ is square integrable the
barycenter can be alternatively characterized by
\begin{equation}\label{E:3.2}
b(q)=\mbox{arg}\min_z \int_M d^2(z, x)q(dx).
\end{equation}
\begin{remark}\label{R:3.3}
For the case that $q=\sum_{i=1}^n w_i\delta_{x_i}$, where $(w_1,\ldots,w_n)$ is a weight and
$\delta_{x_i}$ is the point mass at $x_i$, we have
$$b(q)=\mbox{arg}\inf_z \int_M d^2(z,x) q(dx)=\mbox{arg}\inf_z\sum_{i=1}^n w_id^2(z,x_i)={\frak G}_n(w_1,\ldots,w_n;x_1,\ldots,x_n).$$
Thus in this case $q$ is square integrable and its barycenter $b(q)$ agrees with the weighted least squares mean of $(x_1,\ldots,x_n)$.
\end{remark}

For $X:\Omega\to M$ integrable, we define its \emph{expected value} $EX$ by
\begin{eqnarray}\label{E:3.3}
EX&=&\mbox{arg}\inf_{z\in M}\int_\Omega \left[d^2(z,X(\omega))-d^2(y,X(\omega))\right]\sigma(d\omega)\\
&=&\mbox{arg}\min_z \int_M [d^2(z, x)-d^2(y,x)]q_X(dx)=b(q_X).  \notag
\end{eqnarray}
From this definition it is clear that integrable i.d.\ random variables have the same expectation.

It is also possible to define and prove notions of a Law of Large Numbers for a sequence of i.i.d.\ random variables
into a metric space $M$.  Let $\{X_n:n\in\mathbb{N}\}$ be a sequence of independent, identically distributed random variables on
some probability space $(\Omega,\mathcal{A},\sigma)$ into $M$.    Let $\mu_n$ be an $n$-mean
on $M$ for each $n$, for example one obtained by the symmetrization procedure or least squares.  We use these
means to form the ``average" $Y_n$ of  the given random variables according to the rule
$Y_n(\omega):=\mu_n(X_1(\omega),\ldots,X_n(\omega))$.  Now under suitable hypotheses Es-Sahib and Heinrich
\cite{EH99} and Sturm \cite{St03} show that a strong law of large numbers is satisfied, that is, the $Y_n$ converge pointwise
a.e.\ to a common point $b$.  The principal result of Sturm \cite[Theorem 4.7]{St03} is crucial for our purposes.
\begin{theorem}\label{T:Sturm}
Let $\{X_n\}_{n\in\mathbb{N}}$ be a sequence of uniformly bounded i.i.d.\ random variables from a probability space
$(\Omega,\mathcal{A},\sigma)$ into an NPC space $M$.  Let $S_n$ denote the inductive mean for each $n\geq 2$, and
set $Y_n(\omega)=S_n(X_1(\omega),\ldots,X_n(\omega))$.  Then $Y_n(\omega)\to EX_1$ as $n\to\infty$ for almost all $\omega\in \Omega$.
\end{theorem}

\section{A basic construction}
In this section we specialize Theorem \ref{T:Sturm} to the case of finitely supported
probability measures, the case of interest to us.
We first recall a standard construction of probability
theory.  For each $k\in\mathbb{N}$, let $\Omega_k$ denote a copy of
an $n$-element set labelled $\{\xi_1,\ldots,\xi_n\}$  equipped with the probability
measure $\sigma_k=\sum_{i=1}^n w_i\delta_{\xi_i}$, where $(w_1,\ldots,w_n)$ is
a weight.  We set
$\Omega=\prod_{k=1}^\infty \Omega_k$.  We call a subset of $\Omega$
a \emph{box} if it is of the form $A=\prod_{k=1}^\infty A_k$, where
$\emptyset\ne A_k\subseteq \Omega_k$ and set
$\sigma(A)=\prod_{k=1}^\infty \sigma_k(A_k)$.  It is a standard
result of measure theory that $\sigma$ uniquely extends to a
probability measure, called the \emph{product measure} and again
denoted $\sigma$, on the $\sigma$-algebra  $\mathcal{A}$ generated
by the boxes.

Now let $(M,d)$ be an NPC space, let $\{x_1,\ldots,x_n\}\subseteq
M$, and let $\Bw=(w_1,\ldots,w_n)$ be a weight.  For each positive
integer $k$, let $X_k:\Omega \to M$ be  defined by $X_k(\omega)=x_i$
if $\pi_k(\omega)=\xi_i$, where $\pi_k:\Omega\to\Omega_k$ is
projection into the $k^{th}$-coordinate.   It is straightforward to
verify from the definition of the product measure that the sequence
$\{X_k\}$ of random variables is independent.  Furthermore, each
$X_k$ has distribution $\sum_{i=1}^n w_i\delta_{x_i}$, and hence the
sequences are identically distributed.

We define $Y_k:\Omega\to M$ for each $k$ by
$Y_k(\omega)=S_k(X_1(\omega),\ldots,X_k(\omega))$, where $S_k$ is
the inductive mean.    By Theorem \ref{T:Sturm} we have that
$\lim_{k\to\infty} Y_k(\omega)= EX_1=b(q_{X_1})$ a.e. From Remark
\ref{R:3.3} it follows that that $\lim_{k\to\infty} Y_k(\omega)
={\frak G}_n(\Bw;x_1,\ldots, x_n)$
a.e.  We summarize this special case of Theorem \ref{T:Sturm}.
\begin{corollary}\label{C:Sturm}
Let $(M,d)$ be an NPC space, let $\{x_1,\ldots,x_n\}\subseteq M$,
and let $\Bw=(w_1,\ldots,w_n)$ be a weight.
Then $\lim_{k\to\infty} Y_k(\omega)={\frak G}_n(\Bw;x_1,\ldots, x_n)$ a.e.\
for the $\{Y_k\}$ given in the preceding construction.
\end{corollary}

We consider a basic example.
\begin{proposition} \label{P:hilb} Let $\mathcal{H}$ be a Hilbert space endowed with
the metric induced by the inner product.  Then
\begin{itemize}
\item[(i)] $\mathcal{H}$ is an NPC space.
\item[(ii)] The binary $t$-weighted mean of $x$ and $y$ is given by $(1-t)x+ty$.
\item[(iii)]The inductive mean is given by  $S_n(x_1,\ldots,x_n)=\sum_{i=1}^n (1/n)x_i$.
\item[(iv)]  The weighted least squares mean for weight $\Bw=(w_1,\ldots, w_n)$ is given
by ${\frak G}_n(\Bw;x_1,\ldots,x_n)=\sum_{i=1}^n w_ix_i$.
\item[(v)] For $\{X_k\}_{k\in\mathbb{N}}$ and a weight $\Bw=(w_1,\ldots,w_n)$ as
given in the preceding construction, we have a.e.
$$\lim_{k\to\infty} \sum_{i=1}^k (1/k) X_i(\omega)\to  {\frak G}_n(\Bw;x_1,\ldots,x_n)=\sum_{i=1}^n w_ix_i.$$
\end{itemize}
\end{proposition}

\begin{proof} (i) It is standard that Hilbert spaces satisfy the parallelogram law, hence the semiparallelogram law $(\ref{E:2.1})$,
and hence are NPC spaces (see e.g.\ \cite[Proposition 3.5]{St03}).

(ii) The map on $[0,1]$ given by $t\mapsto (1-t)x+ty$ is a metric geodesic taking $0$ to $x$ and $1$ to $y$.
Since such geodesics are unique in NPC spaces, it must give the $t$-weighted mean.

(iii) By definition and induction $$S_n(x_1,\ldots,x_n)=\frac{n-1}{n}S_{n-1}(x_1,\ldots, x_{n-1})+
\frac{1}{n} x_n=\frac{n-1}{n}\sum_{i=1}^{n-1}\frac{1}{n-1}x_i+ \frac{1}{n} x_n=\sum_{i=1}^n\frac{1}{n} x_i.$$

(iv) Consider the measure $q=\sum_{i=1}^nw_i\delta_{x_i}$.   Then for any $y\in\mathcal{H}$,
$$\left\langle {\frak G}_n(\Bw;x_1,\ldots, x_n),y\right\rangle=\langle b(q),y\rangle=\int_{{\mathcal H}} \langle x,y\rangle q(dx)=
\sum_{i=1}^n w_i\langle x_i,y\rangle=\left\langle \sum_{i=1}^n w_ix_i,y\right\rangle,$$
where the first equality follows from Remark \ref{R:3.3} and the second is the content of \cite[Proposition 5.4]{St03}.
The conclusion of (iv) is now immediate.

(v) In the earlier construction of this section we have $Y_k(\omega)=\sum_{i=1}^k\frac{1}{k} X_i(\omega)$ by
part (iii). The conclusion of (v) then follows from Corollary \ref{C:Sturm} and (iv).
\end{proof}

\section{Monotonicity and Loewner-Heinz NPC spaces}
The fundamental Loewner-Heinz inequality for positive definite matrices asserts that $A^{1/2}\leq B^{1/2}$ whenever
$A\leq B$.  This can be written alternatively as $A\#I\leq B\#I$ whenever $A\leq B$ and extends to the equivalent
monotonicity property that $A_1\# A_2\leq B_1\# B_2$ whenever $A_1\leq  B_1$ and $A_2\leq B_2$.  These considerations
motivate the next definition.
\begin{definition}
A \emph{Loewner-Heinz NPC space} is an NPC space equipped with a
closed partial order $\leq$  satisfying $x_1\#x_2\leq y_1\#y_2$
whenever $x_i\leq y_i$ for $i=1,2$.

\end{definition}
A mean $\mu:M^n\to M$ on a partially ordered metric space is called
\emph{order-preserving} or \emph{monotonic}
if $x_i\leq y_i$ for $i=1,\ldots, n$ implies $\mu(x_1,\ldots, x_n)\leq
\mu(y_1,\ldots, y_n)$.

\begin{lemma}\label{L:4.1} The inductive mean $S_{n}$ on a Loewner-Heinz
NPC space  is monotonic for every $n\geq 2$.
\end{lemma}

\begin{proof} We first observe that $x_1\#_t x_2\leq y_1\#_t y_2$ whenever
$x_i\leq y_i$ for $i=1,2$ by the standard argument of extending the
inequality to the dyadic weighted means by induction for the case of
the dyadic rationals, and then extending to general $t\in [0,1]$ by
continuity in $t$ and the closedness of the relation $\leq$.
Assuming that the inductive $k$-mean $S_k$ is monotonic, it follows
that $S_{k+1}(x_1,\ldots, x_{k+1})=S_k(x_1,\ldots,
x_k)\#_{\frac{1}{k+1}} x_{k+1}$ is monotonic since $S_k$ and the
$t$-weighted mean both are.
\end{proof}

\begin{theorem}\label{T:mono}
Let $(M,d,\leq)$ be a Loewner-Heinz NPC space.  Then for a fixed weight $\Bw=(w_1,\ldots,w_n)$
the weighted least squares mean ${\frak G}_n$ is monotonic for $n\geq 2$.
\end{theorem}

\begin{proof}  Assuming $x_i\leq y_i$ for $1\leq i\leq n$, we show ${\frak G}_n(\Bw;x_1,\ldots, x_n)\leq
{\frak G}_n(\Bw;y_1,\ldots, y_n)$, where ${\frak G}_n$ is the least squares
mean on $M^n$.  Let $\Omega_k$ be a copy of the $n$-element set
$\{\xi_1,...,\xi_n\}$ equipped with the measure $\sum_{i=1}^n
w_i\delta_{\xi_i}$.  Let $\Omega=\prod_{k=1}^\infty \Omega_k$ be
the countable product of the $\Omega_k$ with the product measure.
Let $X_k:\Omega \to M$ be  defined by $X_k(\omega)=x_i$ if
$\pi_k(\omega)=\xi_i$, where $\pi_k:\Omega\to\Omega_k$ is projection
into the $k^{th}$-coordinate.  Similarly we define $\tilde
X_k:\Omega\to M$ by $\tilde X_k(\omega)= y_i$ if
$\pi_k(\omega)=\xi_i$. As we have seen in the previous section
$\{X_k\}$ is i.i.d.\ with distribution  $\sum_{i=1}^n w_i\delta_{x_i}$, while
$\{\tilde
X_k\}$ is i.i.d.\ with  distribution $\sum_{i=1}^n w_i\delta_{y_i}$.
 Finally we note that
$(X_1(\omega),\ldots,X_k(\omega))$ is coordinatewise less than or
equal to $(\tilde X_1(\omega),\ldots,\tilde X_k(\omega))$ since
$x_i\leq y_i$ for each $i=1,\ldots, n$.

We define $Y_k, \tilde Y_k:\Omega\to M$ by
$Y_k(\omega)=S_k(X_1(\omega),\ldots,X_k(\omega))$ and
 $\tilde Y_k(\omega)=\newline S_k(\tilde X_1(\omega),\ldots,\tilde X_k(\omega))$.  It follows from Lemma
\ref{L:4.1}  that $Y_k(\omega)\leq\tilde Y_k(\omega)$ for each
$\omega\in\Omega$.  By Corollary  \ref{C:Sturm} we have that
$\lim_{k\to\infty} Y_k ={\frak G}_n(\Bw;x_1,\ldots,x_n)$ a.e. and
$\lim_{k\to\infty} \tilde Y_k={\frak G}_n(\Bw;y_1,\ldots,y_n)$ a.e.
By the closedness of the partial order (and the fact that the
intersection of two sets of measure $1$ still has measure $1$), we
conclude that  ${\frak G}_n(\Bw;x_1,\ldots, x_n)\leq {\frak
G}_n(\Bw;y_1,\ldots, y_n)$.
\end{proof}

Since the trace metric on the space ${\Bbb P}$ of $m\times m$
positive definite (real or complex) matrices makes it a Loewner-Heinz NPC
space with respect to the Loewner order (see e.g.\ \cite{LL01}), we
have the following corollary.

\begin{corollary}\label{C:mono}
The weighted least squares mean on the set $\mathbb{P}$ of positive definite matrices is
monotonic.
\end{corollary}

\begin{remark}
Loewner \cite{Loe}  proved that a function defined on an open
interval is operator monotone if and only if it allows an analytic
continuation into the complex upper half-plane with nonnegative
imaginary part. The function $ f(t) = t^{\alpha}, \alpha\in [0,1] $
is operator monotone on the positive reals,  that is, $ X\leq Y$
implies $X^{\alpha}\leq Y^{\alpha}$ for positive definite matrices
$X$ and $Y.$   The inequality was independently proved by Heinz
\cite{Heinz}. It is equivalent to the extended monotonicity property
of the weighted geometric mean: $B_{1}\#_{t}B_{2}\leq
A_{1}\#_{t}A_{2}, t\in [0,1],$ whenever $B_{1}\leq A_{1}$ and
$B_{2}\leq A_{2}.$ It is natural to consider the monotonicity of the
least squares mean $\G_{n}(\omega;{\Bbb B})\leq \G_{n}(\omega;{\Bbb
A})$ whenever $B_i\leq A_i$ for each $i$ as an $n$-variable
Loewner-Heinz inequality for positive definite matrices.
\end{remark}

A function $F:{\Bbb P}^n\to {\Bbb P}$ is \emph{jointly
concave} if for any $(A_1,\dots, A_n),(B_1,\ldots,B_n)\in {\Bbb
P}^n$ and $0\leq t\leq 1$, we have
$$ t F(A_1,\ldots,A_n)+(1-t)F(B_1,\ldots, B_n)\leq F(tA_1+(1-t)B_1,\ldots, tA_n+(1-t)B_n).$$
\begin{proposition}\label{P:cave}
The least squares mean ${\frak G}_n:{\mathbb{P}}^n\to
{\mathbb{P}}$ for the trace metric is jointly concave for each
$n\geq 2$.
\end{proposition}
\begin{proof}
It is a standard result that the two-variable weighted geometric
mean on ${\Bbb P}$ is jointly concave.  It follows directly by
induction that the inductive mean $S_n$ of positive definite
matrices is jointly concave for $n\geq 2$.

Fix  $(A_1,\dots, A_n),(B_1,\ldots,B_n)\in {\Bbb P}^n$ and a weight $\Bw=
(w_1,\ldots,w_n)$.
Construct random variables $\{X_k\}$, $\{\tilde X_k\}$ as in the proof of
Theorem \ref{T:mono}
with $A_i$ replacing $x_i$ and $B_i$
replacing $y_i$ for each $i$. For $Y_k= S_k(X_1,\ldots, X_k)$ and
$\tilde Y_k=S_k(\tilde X_1,\ldots,\tilde X_k\}$, we conclude from
the concavity of $S_k$ that
\begin{eqnarray*}
tY_k+(1-t)\tilde Y_k&\leq& S_k(tX_1+(1-t)\tilde X_1,\ldots,
tX_k+(1-t)\tilde X_k)=S_k(Z_1,\ldots, Z_k),
\end{eqnarray*}
where $Z_i=tX_i+(1-t)\tilde X_i$ for $1\leq i\leq k$.  Note that the $Z_k$ are i.i.d.\ with
each $Z_k$ having distribution the probability measure that assigns mass $w_i$ to
each $tA_i+(1-t) B_i$, $1\leq i\leq n$.
From Corollary \ref{C:Sturm} the limit of both sides exists a.e. and is given by the appropriate
least squares mean, and from the closedness of the order we conclude
$$t{\frak G}_n(\Bw;A_1,\ldots,A_n)+(1-t){\frak G}_n(\Bw;B_1,\ldots,B_m)\leq {\frak G}_n(\Bw;Z_1,
\ldots, Z_n),$$
where $Z_i=tA_i+(1-t)B_i$ for each $i$.
\end{proof}

\section{Other properties of the least squares mean}
The fact that the unweighted least squares mean is bounded above by
the arithmetic mean, and hence below by the
harmonic mean has been recently shown by Yamazaki \cite{Ya}.  We
give an alternative approach via probabilistic methods and derive the
result for the weighted least squares mean.
\begin{proposition}\label{P:geoarith}
For $(A_1,\dots, A_n)\in {\Bbb P}^n$ and a weight $\Bw=
(w_1,\ldots,w_n)$, we have
$$\left(\sum_{i=1}^{n}w_{i}A_{i}^{-1}\right)^{-1}\leq {\frak G}_n(\Bw;A_1,\ldots,A_n)\leq \sum_{i=1}^n w_iA_i.$$
\end{proposition}
\begin{proof}
It is a standard result that the two-variable weighted geometric
mean on ${\mathbb{P}}$ is below the corresponding weighted
arithmetic mean: $A\#_t B\leq (1-t)A+tB$  for $0\leq t\leq 1$. It
follows by induction that the inductive mean satisfies for each $k$
\begin{eqnarray*}
S_k(B_1,\ldots,B_{k}) &=& S_{k-1}(B_1,\ldots, B_{k-1})\#_{1/k} B_{k}\\
&\leq & \frac{k-1}{k}\sum_{i=1}^{k-1}\frac{1}{k-1} B_i +\frac{1}{k}
B_k= \frac{1}{k}\sum_{i=1}^k B_i.
\end{eqnarray*}

Construct a sequence of i.i.d.\ random variables $\{X_k\}$ as in
Section 5 such that  the distribution is $\sum_{i=1}^n
w_i\delta_{A_i}$ for each $X_k$.   Set $Y_k= S_k(X_1,\ldots, X_k)$
and
 for each $k$.   From Corollary \ref{C:Sturm}
$\lim_{k\to \infty} Y_k(\omega)={\frak G}_n(\Bw;A_1,\ldots, A_n)$ a.e.

Endow the space of Hermitian matrices ${\Bbb H}$ containing
$\mathbb{P}$ with the Hilbert space structure with inner product
$\langle A,B\rangle=\mbox{tr}A^{*}B$.  Then $\mathbb{P}$ is an open
subspace of ${\Bbb H}$.  Set $Z_k=\sum_{i=1}^k (1/k) X_i$, where
$\{X_k\}$ are the random variables of the previous paragraph.  By
Proposition \ref{P:hilb}9(v), $\lim_{k\to\infty} Z_{k}(\omega)\to
\sum_{i=1}^n w_iA_i$ a.e.  By the first paragraph $Y_k(\omega)\leq
Z_k(\omega)$ for all $k,\omega$.   From the closure of the order, we
conclude that ${\frak G}_n(\Bw;A_1,\ldots,A_n)\leq \sum_{i=1}^n
w_iA_i$.

The first inequality in the conclusion of the proposition follows from the second and the
fact that inversion is an isometry for the trace metric and hence preserves the least squares
mean.
\end{proof}

Let $M$ be an NPC space.  Given probability measures $p,q\in\mathcal{P}(M)$, we
say that a  probability measure $\mu\in \PP(M^2)$ is a \emph{coupling} of $p$ and $q$
if the  \emph{marginals} of $\mu$ are $p$ and $q$, that is, if for all Borel sets
$B\in\mathcal{B}(M)$
\begin{equation}
\mu(B\times M)=p(B) \mbox{ and } \mu(M\times B)=q(B).
\end{equation}
\begin{definition}
The ($L^1$)-\emph{distance} $\rho$ on $\PP^1(M)$  is given by
$$W(p,q)=\inf\left\{ \int_{M\times M} d(x,y) \mu(dxdy):\mu \mbox{ is a coupling of }p\mbox{ and }q\right\}.$$
 We adopt the most common name for the metric, the \emph{Wasserstein} distance, although it also appears under
 a variety of other names such as the \emph{Kantorovich-Rubenstein} distance.
\end{definition}

\begin{proposition}\label{P:5.2}
For $(x_1,\ldots,x_n),(y_1,\ldots,y_n)\in M^n$, a weight $\Bw=(w_1,\ldots,w_n)$,
 and the corresponding finitely supported probability measures
$q_1=\sum_{k=1}^n w_i\delta_{x_i}$  and $q_2=\sum_{k=1}^n w_i\delta_{y_i}$ on $M$,
$$d({\frak G}_n(\Bw;x_1,\ldots,x_n), {\frak G}_n(\Bw;y_1,\ldots,y_n))\leq W(q_1,q_2)\leq \sum_{i=1}^ n w_id(x_i,y_i).$$
Hence, in particular, the least squares mean ${\frak G}_n$ is continuous for each $\Bw$.
\end{proposition}
\begin{proof}
Define $\mu$ on $M\times M$ by $\mu=\sum_{i=1}^n w_i
\delta_{(x_i,y_i)}$.  One sees readily that $\mu$ is a coupling of
$p$ and $q$, and thus $W(p,q)\leq \int_{M\times M}
d(x,y)\mu(dxdy)=\sum_{i=1}^n w_id(x_i,y_i).$ By Theorem 6.3 of
\cite{St03}, the barycentric map $b:\PP^1(M)\to M$ satisfies for all
$p,q$ the fundamental contraction property $d(b(p),b(q))\leq
W(p,q)$. By Remark \ref{R:3.3} $b(q_1)={\frak
G}_n(\Bw;x_1,\ldots,x_n)$ and similarly $b(q_2)={\frak
G}_n(\Bw;y_1,\ldots,y_n)$ . Thus
$$d({\frak G}_n(\Bw;x_1,\ldots,x_n), {\frak G}_n(\Bw;y_1,\ldots,y_n))=d(b(q_1),b(q_2))\leq W(q_1,q_2)\leq \sum_{i=1}^ n w_i d(x_i,y_i).$$
The fact that the right hand of the preceding is larger than the left hand directly establishes the continuity of ${\frak G}_n$.
\end{proof}

From this result together with Corollary \ref{C:mono} and Propositions \ref{P:cave} and \ref{P:geoarith} we
conclude that the least squares mean of positive definite matrices
satisfies the continuity, monotonicity, joint concavity, and AGM inequality properties, and
hence all the fundamental properties of the
 geometric means of positive definite matrices defined for and satisfied  by the ALM and BMP constructions
 \cite{ALM04,BMP10}; see \cite{Ya} for other properties.

\section{Appendix: The least squares mean on symmetric cones}
In this section, we shall see that the techniques and results from
the probabilistic treatment of the least squares mean for positive
definite matrices carry over, typically with little change, to the
case of symmetric cones. We first briefly describe (following mostly
\cite{FK}) some Jordan-algebraic concepts pertinent to our purpose.
A {\it Jordan algebra} $V$ over ${\Bbb R}$ is a finite-dimensional
commutative algebra with identity $e$ satisfying
$x^{2}(xy)=x(x^{2}y)$ for all $x,y\in V.$ For $x\in V,$ let $L(x)$
be the linear operator defined by $L(x)y=xy,$ and let
$P(x)=2L(x)^{2}-L(x^{2}).$ The map $P$ is called the quadratic
representation of $V.$ An element $x\in V$ is said to be invertible
if there exists an element $x^{-1}$ in the subalgebra generated by
$x$ and $e$ such that $xx^{-1}=e.$


An element $c\in V$ is called an idempotent if $c^{2}=c.$ We say
that $c_{1},\dots,c_{k}$ is a complete system of orthogonal
idempotents if $c_{i}^{2}=c_{i}, c_{i}c_{j}=0, i\neq j,
c_{1}+\cdots+c_{k}=e.$  An idempotent is primitive if it is non-zero
and cannot be written as the sum of two non-zero idempotents. A
Jordan frame is a complete system of primitive idempotents.

A Jordan algebra $V$ is said to be {\it Euclidean} if there exists
an inner product $\langle \cdot,\cdot\rangle$ such that
for all $x,y,z\in V:$
\begin{eqnarray}\label{E:eu}\langle xy,z\rangle=\langle y,xz\rangle. \end{eqnarray}
The following spectral theorem for Euclidean Jordan algebras
appears in \cite{FK}.
\begin{theorem}\label{T:se} Any two Jordan frames in an Euclidean Jordan algebra
$V$ have the same number of elements $($called the rank of $V$,
denoted ${\mathrm{rank}}(V)$$)$. Given $x\in V,$ there exists a
Jordan frame $c_{1}, \dots, c_{r}$ and real numbers
$\lambda_{1},\dots, \lambda_{r}$ such that
$$x=\sum_{i=1}^{r}\lambda_{i}c_{i}.$$
\end{theorem}

\begin{definition} Let $V$ be a Euclidean Jordan algebra of
${\mathrm{rank}}(V)=r.$  The spectral mapping $\lambda:V\to {\Bbb
R}^{r}$ is defined by $\lambda(x)=(\lambda_{1}(x),\dots,
\lambda_{r}(x)),$ where the $\lambda_{i}(x)$'s are eigenvalues of $x$
$($with multiplicities$)$ as in Theorem $\ref{T:se}$ in
non-increasing order $\lambda_{{\mathrm{max}}}(x)=\lambda_{1}(x)\geq
\lambda_{2}(x)\geq \cdots\geq
\lambda_{r}(x)=\lambda_{\mathrm{min}}(x).$

We define ${\mathrm{det}}(x)=\prod_{i=1}^{r}\lambda_{i}(x)$ and $
{\mathrm{tr}}(x)=\sum_{i=1}^{r}\lambda_{i}(x).$ Then ${\mathrm{tr}}$
is a linear form on $V$ and ${\mathrm{det}}$ is a homogeneous
polynomial of degree $r$ on $V.$
\end{definition}

The trace inner product $\langle x,y\rangle={\mathrm{tr}}(xy)$ in a
Euclidean Jordan algebra satisfies (\ref{E:eu}). We will assume that
$V$ is a  Euclidean Jordan algebra of rank $r$  and equipped with
the trace inner product  $\langle x,y\rangle={\mathrm{tr}}(xy).$ Let
$Q$ be the set of all square elements of $V.$ Then $Q$ is a closed
convex cone of $V$ with $Q\cap -Q=\{0\},$ and is the set of elements
$x\in V$ such that $L(x)$ is positive semi-definite.  It turns out
that $Q$ has non-empty interior $\Omega,$ and $\Omega$ is a
symmetric cone, that is, the group $G(\Omega)=\{g\in
{\mathrm{GL}}(V)| g(\Omega)=\Omega\}$ acts transitively on it and
$\Omega$ is a self-dual cone with respect to the inner product
$\langle \cdot|\cdot\rangle.$ Furthermore,  for any $a$ in $\Omega$,
$P(a)\in G(\Omega)$ and is positive definite. We note that any
symmetric cone (self-dual, homogeneous open convex cone) can be
realized as an interior of squares in an appropriate Euclidean
Jordan algebra \cite{FK}.

We remark  that
${\mathrm{det}}(P(a^{1/2})b)={\mathrm{det}}(a){\mathrm{det}}(b)$ for
all $a,b\in \Omega$ and ${\mathrm{tr}}(\log a)=\log
{\mathrm{det}}(a)$ for all $a\in \Omega,$ where
$a^{1/2}=\sum_{i=1}^{r}\lambda_{i}^{1/2}c_{i},\ \log
a=\sum_{i=1}^{r}(\log \lambda_{i})c_{i},$ and $
a=\sum_{i=1}^{r}\lambda_{i}c_{i}$ a spectral decomposition of $a$
(\cite{FK}).

\begin{proposition} The symmetric cone $\Omega\subseteq V$ has the following
properties:
$$\Omega=\{x^{2}: x\ {\mathrm{is\ invertible}}\}=\{x: L(x)\ {\mathrm{is\ positive\
definite}}\}=\{x: \lambda_{\min}(x)>0\}.$$
\end{proposition}

The space ${\Bbb H}_{m}$ of $m\times m$ Hermitian matrices equipped
with the trace inner product $\langle
X,Y\rangle={\mathrm{tr}}(X^{*}Y)$ and the Jordan product $X\circ
Y=\frac{1}{2}(XY+YX)$ is a typical example of Euclidean Jordan
algebras. In this case the corresponding symmetric cone is ${\Bbb
P}_{m},$ the convex cone of $m\times m$ positive definite Hermitian
matrices, and the quadratic representation is given by $P(X)Y=XYX.$

It turns out \cite{FK} that the symmetric cone $\Omega$  admits a
$G(\Omega)$-invariant Riemannian metric  defined by $\langle
u,v\rangle_{x}=\langle P(x)^{-1}u,v\rangle, x\in \Omega, u,v\in V.$
The inversion $j(x)=x^{-1}$ is  an  involutive isometry fixing $e.$
It is a symmetric Riemannian space of non-compact type and hence is
an ${\mathrm {NPC}}$ space with respect to its distance metric
\cite{Lang,LL01}. The isometry properties just mentioned give symmetric
cone analogs of properties (P6) and (P8) for the weighted least
squares mean.  Permutation invariance (P3) of the least squares mean
holds in any metric space.

The unique geodesic curve joining
$a$ and $b$ is $t\mapsto a\#_{t}b:=P(a^{1/2})(P(a^{-1/2})b)^{t}$
 and the Riemannian distance $d(a,b)$ is given by
$d(a,b)=\left(
\sum_{i=1}^{r}\log^{2}\lambda_{i}(P(a^{-1/2})b)\right)^{1/2}.$ See
\cite{Lang,LL01,Lim} for more details.
 The geodesic middle (geometric mean) of $a$ and $b$ is given by
$a\#b:=a\#_{1/2}b=P(a^{1/2})(P(a^{-1/2})b)^{1/2}.$

In \cite{Lim,Lim1}, it is shown that the geometric mean is monotone
for the cone ordering, $x\leq y$ if and only if $y-x\in {\overline
\Omega}$, and therefore we conclude that every symmetric cone is a
Loewner-Heinz NPC space. By Theorem \ref{T:mono} and Proposition
\ref{P:5.2}, the weighted least squares mean on a symmetric cone is
monotonic and non-expansive, in particular continuous.  Hence
properties (P4) and (P5) are satisfied.

The AGH-inequality for the two-variable case can be reduced to the case of
two elements sharing a ``diagonalization" over some Jordan frame, in which
case the inequality follows from the real number case.
The proof of Proposition \ref{P:geoarith} then establishes (P10)
 for the general weighted least squares mean.

 We say that two
 elements $a$ and $b$
\emph{commute} if they share the same Jordan frame. Then $(P1)$ follows
easily  for the least squares mean.

The properties $(P2)$ and $(P9)$ follow from the characterization
 of the least squares mean as the unique member of $\Omega$ satisfying
 $\sum_{i=1}^{n}w_{i}\log(P(x^{1/2})a_{i}^{-1})=0$ (which follows from  a
 standard method for computing the Hessian operator of the distance
 function on Riemannian manifolds).   If  $z={\frak B}_{n}({\bf
 w};x_{1},\dots,x_{n}),$ then
 $\sum_{i=1}^{n}w_{i}\log P(z^{1/2})x_{i}=0.$
Setting  $y=\alpha_{1}^{w_{1}}\cdots \alpha_{n}^{w_{n}}z$ we have
 \begin{eqnarray*}
 \sum_{i=1}^{n}w_{i}\log(P(y^{1/2})(\alpha_{i}x_{i})^{-1})&=&\sum_{i=1}^{n}w_{i}\log\frac{\prod_{i=1}^{n}\alpha_{i}^{w_{i}}}{\alpha_{i}}
 P(z^{1/2})x_{i}^{-1}\\
 &=&\sum_{i=1}^{n}w_{i}\left(\log
 \frac{\prod_{i=1}^{n}\alpha_{i}^{w_{i}}}{\alpha_{i}} e+ \log
 P(z^{1/2})x_{i}^{-1}\right)\\
 &=&\sum_{i=1}^{n}w_{i}\log
 \frac{\prod_{i=1}^{n}\alpha_{i}^{w_{i}}}{\alpha_{i}} e=0
 \end{eqnarray*}
 where  the second equality follows from the fact that $\log ta=\log
 t e+\log a$ for any $t>0$ and $a\in \Omega.$ This establishes
 $(P2)$. The determinantal identity follows by considering the trace
 functional:
 \begin{eqnarray*}
 0&=&{\mathrm{tr}}\left(\sum_{i=1}^{n}w_{i}\log
 P(z^{1/2})a_{i}^{-1}\right)=\sum_{i=1}^{n}w_{i}{\mathrm{tr}}(\log
 P(z^{1/2})a_{i}^{-1})\\&=&\sum_{i=1}^{n}w_{i}\log{\mathrm{det}}(
 P(z^{1/2})a_{i}^{-1})=\sum_{i=1}^{n}w_{i}\log{\mathrm{det}}(z){\mathrm{det}}(a_{i}^{-1})\\
 &=&\sum_{i=1}^{n}w_{i}\log{\mathrm{det}}(z)-\sum_{i=1}^{n}w_{i}\log{\mathrm{det}}(a_{i})=
\log{\mathrm{det}}(z)-\log
 \prod_{i=1}^{n}{\mathrm{det}}(a_{i})^{w_{i}}.
 \end{eqnarray*}

 The joint concavity (P7) of the two-variable geometric mean  is as yet unknown
 for general symmetric cones (unlike the positive definite matrix case)
 and hence we do not yet have the joint concavity property for the weighted
 least squares mean.

\end{document}